\documentclass[12pt]{amsart}
\usepackage{amsmath}
\usepackage{amsfonts}
\usepackage{amssymb}
\usepackage{latexsym}
\usepackage{amsthm}
\usepackage{a4wide}

\newtheorem{mtheorem}{Theorem}
\newtheorem{theorem}{Theorem}[section]
\newtheorem{lemma}[theorem]{Lemma}

\newtheorem{proposition}[theorem]{Proposition}

\newcounter{other}            
\newtheorem{otherl}[other]{ Lemma}        

\newcommand{\Cn}{\mathbb{C}^n}

\newcommand{\Bn}{\mathbb{B}_ n}

\def\a{\alpha}

\numberwithin{equation}{section}
\begin{document}

\title[Schatten class Hankel operators]{Characterization of Schatten class Hankel operators on weighted Bergman spaces}

\author[Jordi Pau]{Jordi Pau}
\address{Jordi Pau \\Departament de Matem\`{a}tica Aplicada i Analisi\\
Universitat de Barcelona\\
08007 Barcelona\\
Catalonia\\
Spain} \email{jordi.pau@ub.edu}



%


\thanks{The author was
 supported by  DGICYT grant MTM$2011$-$27932$-$C02$-$01$
(MCyT/MEC) and the grant 2014SGR289 (Generalitat de Catalunya)}


\begin{abstract}
We completely characterize the simultaneous membership in the Schatten ideals $S_ p$, $0<p<\infty$ of the Hankel operators $H_ f$ and $H_{\overline{f}}$ on the Bergman space, in terms of the behaviour of a local mean oscillation function, proving a conjecture of Kehe Zhu from 1991.
\end{abstract}
\maketitle


\section{Main  results}

\textbf{Problem:} Describe the simultaneous membership in the Schatten ideals $S_ p$ of the Hankel operators $H_ f$ and $H_{\overline{f}}$ acting on weighted Bergman spaces.\\

 The answer given below is the main result of the paper.
\begin{mtheorem}\label{mt}
Let $f\in L^2(\Bn,dv_{\alpha})$ and $0<p<\infty$. The following are equivalent:
\begin{itemize}
\item[(a)] $H_ f$ and $H_{\overline{f}}$ are in $S_ p(A^2_{\alpha},L^2(\Bn,dv_{\alpha}))$.

\item[(b)] $MO_ r(f)\in L^p(\Bn,d\lambda_ n)$ for some (any) $r>0$.
\end{itemize}
\end{mtheorem}
\mbox{}
\\
Here
\[d\lambda_ n(z)=\frac{dv(z)}{(1-|z|^2)^{n+1}}\]
is the M\"{o}bius invariant volume measure on $\Bn$, and $MO_ r(f)$ is a certain type of local mean oscillation function to be defined next after we discuss briefly the history of the problem.

When $f$ is holomorphic on $\Bn$ one has $H_ f=0$, and the membership of $H_{\overline{f}}$ in $S_ p$ is described by $f$ being in the analytic Besov space $B_ p$ if $p>\gamma_ n$, and $f$ constant if $0<p\le \gamma_ n$ \cite{AFP,AFJP,J,Wall,Z0,Z1}. The cut-off point is $\gamma_ n=1$ if $n=1$, and $\gamma_ n=2n$ if $n\ge 2$. The equivalence between (a) and (b) was conjectured (at least for $p\ge 1$) in 1991 by K. Zhu in \cite{Z1}.
It was previously known that, if $p>\frac{2n}{n+1+\alpha}$, then (a) is equivalent to\\

\begin{itemize}
\item[(c)] $MO_{\alpha}(f)\in L^p(\Bn,d\lambda_ n)$,\\
\end{itemize}
 where $MO_{\alpha}(f)$ is a ``global" mean oscillation type function. The equivalence between (a) and (c) for $p>\frac{2n}{n+1+\alpha}$ was proved in several steps: K. Zhu \cite{Z1} proved the case $p\ge 2$; J. Xia \cite{Xia1,Xia2} obtained the case  $\max(1,\frac{2n}{n+1+\alpha})<p\le 2$, and the last case $\frac{2n}{n+1+\alpha}<p\le 1$ has been proved recently by J. Isralowitz \cite{Isr}.  It is also well known that condition (c) can not characterize the membership on the Schatten ideals on the missing range $0<p\le 2n/(n+1+\alpha)$, since on this range, condition (c) implies $f$ is a constant (see \cite[p.233]{Zhu}). Now we recall the concepts and definitions. \\

We denote by $\Bn$ the open unit ball of $\Cn$, and let $dv$ be the usual Lebesgue volume measure on $\Bn$, normalized so that the volume of $\Bn$ is one. We fix a real parameter $\alpha$ with $\alpha>-1$ and write
$dv_{\alpha}(z)=c_{\alpha}\,(1-|z|^2)^{\alpha}dv(z),$
where $c_{\alpha}$ is a positive constant chosen so that $v_{\alpha}(\Bn)=1$. The weighted Bergman space $A^2_{\alpha}:=A^2_{\alpha}(\Bn)$ is the closed subspace of  $L^2_{\alpha}:=L^2(\Bn,dv_{\alpha})$ consisting of holomorphic functions. It is a Hilbert space with inner product
\[ \langle f,g \rangle _{\alpha}=\int_{\Bn} f(z)\,\overline{g(z)}\,dv_{\alpha}(z).\]

The corresponding norm is denoted by $\|f\|_{\a}$. The orthogonal (Bergman) projection $P_{\alpha}:L^2(\Bn,dv_{\alpha})\rightarrow A^2_{\alpha}(\Bn)$ is an integral operator given by
\[
P_{\alpha} f(z)=\int_{\Bn} \frac{ f(w)\,dv_{\alpha}(w)}{(1-\langle z,w\rangle )^{n+1+\alpha}},\qquad f\in L^2(\Bn,dv_{\alpha}).
\]
Given a function $f\in L^2(\Bn,dv_{\alpha})$, the Hankel operator $H_ f$ with symbol $f$ is
\[
H_f =(I-P_{\alpha})M_ f,
\]
where $M_ f$ denotes the operator of multiplication by $f$.
It is well known that the simultaneous study of the Hankel operators $H_ f$ and $H_{\bar{f}}$ is equivalent to the study of the commutator $[M_ f,P_{\alpha}]:=M_ f P_{\alpha}-P_{\alpha} M_ f $ acting on $L^2_{\alpha}$, by virtue of the identity
\[
[M_ f,P_{\alpha}]=\widetilde{H_ f} -(\widetilde{H_{\overline{f}}})^*,
\]
where $\widetilde{H_ f}:=H_ f P_{\alpha}$ acts now on $L^2_{\alpha}$.\\

Let $H$ and $\mathcal{K}$ be separable Hilbert spaces, and let $0<p<\infty$. A
compact operator $T$ from $H$ to $\mathcal{K}$ is said to belong
to the Schatten class $S_ p=S_ p(H,\mathcal{K})$ if its sequence of singular numbers belongs to the
sequence space $\ell^p$ (the singular numbers are the square roots of the eigenvalues of the positive operator $T^*T$, where $T^*$ is the Hilbert adjoint of $T$). For $p\ge 1$, the class $S_ p$ is a Banach space
with the norm $\|T\|_ p=\left (\sum_ n |\lambda_ n|^p\right
)^{1/p},$ while for $0<p<1$ one has \cite[Theorem 2.8]{M} the inequality $\|S+T\|_
p^p\le \|S\|_ p^p+\|T\|_ p^p.$ Also, if $A$ is a bounded operator on $H$, $B$ a bounded operator on $\mathcal{K}$, and $T$ is in $S_ p$, then $BTA$ is in $S_ p$. We refer to \cite[Chapter 1]{Zhu} for a brief account on Schatten classes.\\

For $z\in\Bn$ and $r>0$, the \emph{Bergman metric ball} at $z$ is given by
$
D(z,r)=\big \{w\in\Bn:\,\beta(z,w)<r \big \},
$
where $\beta(z,w)$ denotes the \emph{hyperbolic distance} between $z$ and $w$
induced by the Bergman metric.
If $f$ is locally square integrable with respect to the volume measure on $\Bn$,
the \emph{mean oscillation} of $f$ at the point $z\in \Bn$ in the Bergman metric is
\[
MO_ r(f)(z)=\left [ \frac{1}{v_{\alpha}(D(z,r))}\int_{D(z,r)} \! |f(w)-\widehat{f}_ r(z)|^2\,dv_{\alpha}(w) \right ]^{1/2},
\]
where the averaging function $\widehat{f}_ r$ is given by
\[ \widehat{f}_ r(z)=\frac{1}{v_{\alpha}(D(z,r))}\int_{D(z,r)} \! f(w) \,dv_{\alpha}(w). \]
It is well known \cite{Z1,Zhu} that the simultaneous  boundedness and compactness of the Hankel operators $H_ f$ and $H_{\overline{f}}$ acting on the Bergman space $A^2_{\alpha}$ can be characterized in terms of the properties of the function $MO_ r(f)$. The Hankel operators $H_ f$ and $H_{\overline{f}}$ are both bounded if and only if $MO_ r(f)\in L^{\infty}(\Bn)$; and compact if and only if $MO_ r(f)\in C_ 0(\Bn)$. The same characterization holds using a more ``global" oscillation function that we introduce next.
For any $f\in L^2(\Bn,dv_{\alpha})$ and $z\in \Bn$, let
\[
MO_{\alpha}(f)(z)=\left  [ B_{\alpha}(|f|^2)(z)-|B_{\alpha}(f)(z)|^2 \right ]^{1/2},
\]
where $B_{\alpha}(g)$ denotes the Berezin transform of a function $g\in L^1(\Bn,dv_{\alpha})$ defined as
\[
B_{\alpha} (g)(z)=\langle g k_ z,k_ z \rangle _{\alpha},
\]
where $k_ z$ are the normalized reproducing kernels of $A^2_{\alpha}$, that is, $k_ z=K_ z/\|K_ z\|_{\alpha}$ with $K_ z$ being the reproducing kernel of $A^2_{\alpha}$ at the point $z$, given by
\[
K_ z(w)=\frac{1}{(1-\langle w,z \rangle )^{n+1+\alpha}},\quad w\in \Bn.
\]
In order to prove Theorem \ref{mt}, we must introduce a more general Berezin type transform $B_{\alpha,t}f$, and a more general ``mean oscillation" function  $MO_{\alpha,t}(f)$.
For $\alpha>-1$ and $t\ge 0$, let
\begin{equation}\label{DefKt}
K_ z^t(w)=R^{\alpha,t}K_ z (w)=\frac{1}{(1-\langle w,z \rangle )^{n+1+\alpha+t}}.
\end{equation}
We also denote by $h^t_ z$ to be its normalized function, that is,
$
h^t_ z =K^t_ z/\|K^t_ z\|_{\alpha}.
$
Because $\|K^t_ z \|_{\alpha}\asymp (1-|z|^2)^{-\frac{1}{2}(n+1+\alpha+2t)}$, we have
\[
|h^t_ z(w)| \asymp \frac{(1-|z|^2)^{\frac{1}{2}(n+1+\alpha+2t)}}{|1-\langle w,z \rangle |^{n+1+\alpha+t}}.
\]
If $g\in L^1(\Bn,dv_{\alpha})$, the Berezin type transform $B_{\alpha,t}(g)$ is defined as
\[ B_{\alpha,t}(g) (z)=\langle gh^t_ z, h^t_ z \rangle _{\alpha}.\]
For $f\in L^2(\Bn,dv_{\alpha})$, we also set
\begin{displaymath}
MO_{\alpha,t}(f)(z) =\left (B_{\alpha,t}(|f|^2)(z)-\big |B_{\alpha,t}(f)(z)\big |^{2}\right )^{1/2}.
\end{displaymath}
It is easy to see that
\begin{displaymath}
MO_{\alpha,t}(f)(z)=\big \|f h^t_ z -B_{\alpha,t}(f)(z)\,h^t_ z \big \|_{\alpha},
\end{displaymath}
and that one has also the following double integral expression

\begin{displaymath}
MO_{\alpha,t}(f)(z)^2 =\int_{\Bn}\int_{\Bn} |f(u)-f(w)|^2\,|h^t_ z(u)|^2\,|h^t_ z(w)|^2\, dv_{\alpha}(u)\,dv_{\alpha}(w).
\end{displaymath}
\mbox{}
\\
The idea to use the function $MO_{\alpha,t}(f)$ in the study of Hankel operators has been also suggested by other authors independently (see \cite{Isr,Xia3} for example). We have the following result.

\begin{mtheorem}\label{BMOp}
Let $\alpha>-1$, $r>0$, $f\in L^2(\Bn,dv_{\alpha})$,  and $0<p<\infty$. Then, for each $t\ge 0$ such that  $p>2n/(n+1+\alpha+2t)$, we have
\begin{displaymath}
\int_{\Bn} MO_{\alpha,t}(f)(z)^p\,d\lambda_ n (z) \le C \int_{\Bn}  MO_{r}(f)(z)^p\,d\lambda_ n (z).
\end{displaymath}
\end{mtheorem}

It is easy to check that, for any $z\in \Bn$ and $r>0$,  one has
\[
MO_ r(f)(z)=\left [\frac{1}{2\,v_{\alpha}\big (D(z,r)\big )^2} \int_{D(z,r)}\int_{D(z,r)} |f(u)-f(w)|^2 dv_{\alpha}(u)\,dv_{\alpha}(w)\right ]^{1/2}.
\]
From this expression it follows that the behaviour of the local mean oscillation function $MO_ r(f)$ is independent of the parameter $\alpha$. Also, from this and the double integral expression of $MO_{\alpha,t}(f)$, it is straightforward to see that $MO_ r(f)(z)\le C \,MO_{\alpha,t}(f)(z)$. From this observation and Theorem \ref{BMOp}, we see that Theorem \ref{mt} is equivalent to the following one.

\begin{mtheorem}\label{mt3}
Let $\alpha>-1$, $f\in L^2(\Bn,dv_{\alpha})$ and $0<p<\infty$. The following are equivalent:
\begin{enumerate}
\item[(a)]
 $H_ f$ and $H_{\overline{f}}$ are both in $S_ p(A^2_{\alpha},L^2_{\alpha})$
\item[(b)]  For each (or some) $t\ge 0$ with $p(n+1+\alpha+2t)>2n$, one has $MO_{\alpha,t}(f)\in L^p(\Bn,d\lambda_ n)$.
\end{enumerate}
\end{mtheorem}

From this, it can be seen that the conjecture stated at the end of \cite{Isr} is also true.\\

The paper is organized as follows. After some preliminaries given in Section \ref{S-2}, we prove Theorem \ref{BMOp} in Section \ref{S-3}. All the implications in Theorem \ref{mt} are proved in Section \ref{S-4}, except the necessity in the case $0<p\le 2n/(n+1+\alpha)$. This part is proved in Section \ref{S-6}.  \\

We are not worried on the computation of the exact values of certain constants when are not depending on the important quantities involved, so that we use $C$ to denote a positive constant like that, whose exact value may change at different occurrences, and sometimes we use  the notation $A\lesssim B$ to indicate
that there is a positive constant $C$ such that $A\leq CB$,
and the notation $A\asymp B$ means that both $A\lesssim B$ and $B\lesssim A$ hold.

\section{Some known lemmas}\label{S-2}

We need a well-known
result on decomposition of the unit ball $\Bn$.
A sequence $\{a_k\}$ of points in $\Bn$ is called
a \emph{separated sequence} (in the Bergman metric)
if there exists a positive constant $\delta>0$ such that
$\beta(a_i,a_j)>\delta$ for any $i\neq j$.
By Theorem 2.23 in \cite{ZhuBn},
there exists a positive integer $N$ such that for any $0<r<1$ we can
find a sequence $\{a_k\}$ in $\Bn$ with the following properties:
\begin{itemize}
\item[(i)] $\Bn=\cup_{k}D(a_k,r)$.
\item[(ii)] The sets $D(a_k,r/4)$ are mutually disjoint.
\item[(iii)] Each point $z\in\Bn$ belongs to at most $N$ of the sets $D(a_k,4r)$.
\end{itemize}

Any sequence $\{a_k\}$ satisfying the above conditions  is called
an $r$-\emph{lattice}
in the Bergman metric. Obviously any $r$-lattice is separated.

We need the following well known integral estimate that has become very useful in this area of analysis (see \cite[Theorem 1.12]{ZhuBn} for example).
\begin{otherl}\label{IctBn}
Let $t>-1$ and $s>0$. There is a positive constant $C$ such that
\[ \int_{\Bn} \frac{(1-|w|^2)^t\,dv(w)}{|1-\langle z,w\rangle |^{n+1+t+s}}\le C\,(1-|z|^2)^{-s}\]
for all $z\in \Bn$.
\end{otherl}

We also need the following well known discrete version of the previous lemma.

\begin{otherl}\label{l2}
Let $\{z_ k\}$ be a separated sequence in $\Bn$, and let $n<t<s$.
Then
$$
\sum_{k}\frac{(1-|z_ k|^2)^t}{|1-\langle z,z_ k \rangle |^s}\le C\,
(1-|z|^2)^{t-s},\qquad z\in \Bn.
$$
\end{otherl}

Lemma~\ref{l2} can be deduced from Lemma \ref{IctBn} after noticing that,
if a sequence $\{z_ k\}$ is separated, then there is a constant
$r>0$ such that the Bergman metric balls $D(z_ k,r)$ are pairwise
disjoints.\\

We also need the following version of Lemma \ref{IctBn}, with an extra (unbounded) factor $\beta(z,w)$ in the integrand.
\begin{lemma}\label{Itbeta}
Let $t>-1$ and $s,c>0$. There is a positive constant $C$ such that
\[ I:= \int_{\Bn} \frac{(1-|w|^2)^t\,\beta(z,w)^c\,dv(w)}{|1-\langle z,w\rangle |^{n+1+t+s}}\le C\,(1-|z|^2)^{-s}\]
for all $z\in \Bn$.
\end{lemma}
\begin{proof}
Pick $\varepsilon>0$ so that $t-c\,\varepsilon >-1$ and $s-c\,\varepsilon>0$. Since $\beta(z,w)$ grows logarithmically, we have
\[\beta(z,w)=\beta(0,\varphi_ z(w))\le C (1-|\varphi_ z(w)|^2)^{-\varepsilon}.
\]
Here $\varphi_ z$ denotes the M\"{o}bius transformation sending $z$ to $0$. It follows from the basic identity
\begin{equation}\label{BId}
1-|\varphi_ z(w)|^2=\frac{(1-|w|^2)(1-|z|^2)}{|1-\langle z,w \rangle |^2},
\end{equation}
that
\[
I\lesssim (1-|z|^2)^{-c\varepsilon}\int_{\Bn} \frac{(1-|w|^2)^{t-c\varepsilon}\,dv(w)}{|1-\langle z,w\rangle |^{n+1+t+s-2c\varepsilon}}.
\]
 The desired result then follows from Lemma \ref{IctBn}.
\end{proof}

The corresponding discrete version is stated below.
\begin{lemma}\label{Ic-discretebeta}
Let $\{z_ k\}$ be a separated sequence in $\Bn$. Let $n<t<s$ and $c>0$.
Then
\[
\sum_{k}\frac{(1-|z_ k|^2)^t\,\beta(z,z_ k)^c}{|1-\langle z,z_ k \rangle |^s}\le C\,
(1-|z|^2)^{t-s},\qquad z\in \Bn.
\]
\end{lemma}

We also need the following elementary result.
\begin{lemma}\label{DMOr}
For $r>0$ let $\{a_ k\}$ be an $r$-lattice on $\Bn$. Then
\begin{displaymath}
\sum_ k MO_ r(f)(a_ k)^p \le C_ 1 \int_{\Bn} MO_ {2r}(f)(z)^p\,d\lambda_ n(z)\le C_ 2\sum_ k MO_ {4r}(f)(a_ k)^p
\end{displaymath}
for all $0<p<\infty$.
\end{lemma}
\begin{proof}
It follows from the double integral expression of the mean oscillation  that
\[
MO_ r(f)(a_ k)\le C MO_{2r}(f)(z),\qquad  z\in D(a_ k,r).
\] From this, the result is easily deduced.
\end{proof}

\section{Proof of Theorem \ref{BMOp}}\label{S-3}

Let $\{a_ k\}$ be an $(r/3)$-lattice on $\Bn$. Because $r>0$ is arbitrary, due to  Lemma \ref{DMOr}, it is enough to prove
\begin{equation}\label{T2-Main}
\int_{\Bn} MO_{\alpha,t}(f)(z)^p\,d\lambda_ n(z)\le C \sum_ k MO_ {2r}(f)(a_ k)^p.
\end{equation}

Let $D_ k=D(a_ k,r)$. Then, using the double integral expression of $MO_{\alpha,t}(f)$, we have
\begin{displaymath}
MO_{\alpha,t}(f)(z)^2\le \sum_ {j,k}\int_{D_ j}\int_{D_ k} |f(u)-f(w)|^2\,|h^t_ z(u)|^2\, dv_{\alpha}(u)\,|h^t_ z(w)|^2 dv_{\alpha}(w).
\end{displaymath}
Since $|h^t_ z(u)|\asymp |h^t_ z(a_ k)|$ for $u\in D_ k$ (see estimate (2.20) in p.63 of \cite{ZhuBn}), we get

\begin{displaymath}
\begin{split}
MO_{\alpha,t}(f)(z)^2 & \lesssim  \sum_ {j,k}|h^t_ z(a_ k)|^2\,|h^t_ z(a_ j)|^2\int_{D_ j}\int_{D_ k} |f(u)-f(w)|^2\, dv_{\alpha}(u)\,dv_{\alpha}(w).
\end{split}
\end{displaymath}
Due to the triangle inequality, we see that
\[
MO_{\alpha,t}(f)(z)^2 \lesssim A_ 1(f,z)+A_ 2(f,z),
\]
and because of the symmetry of the terms, in order to establish \eqref{T2-Main} it is enough to show that
\begin{equation}\label{Eq-T2-A1}
\int_{\Bn}A_ 1(f,z)^{p/2}\,d\lambda_ n(z)\lesssim \sum_ j  MO_ {2r}(f)(a_ j)^p
\end{equation}
with
\[
A_ 1(f,z):=\sum_ {j,k}|h^t_ z(a_ k)|^2\,|h^t_ z(a_ j)|^2\,|D_ k|_{\alpha} \int_{D_ j} |f(u)-\widehat{f_ r}(z)|^2\, dv_{\alpha}(u).
\]
Here we use the notation $|D_ k|_{\alpha}=v_{\alpha}(D_ k)\asymp (1-|a_ k|^2)^{n+1+\alpha}$. By Lemma \ref{l2}, we get
\[
A_ 1(f,z)\lesssim \sum_ {j}|h^t_ z(a_ j)|^2 \int_{D_ j} |f(u)-\widehat{f_ r}(z)|^2\, dv_{\alpha}(u),
\]
and by the triangle inequality we have

\begin{equation}\label{T2-EqB0}
A_ 1(f,z)\lesssim A_{11}(f,z)+A_{12}(f,z)
\end{equation}
with
\[
\begin{split}
A_{11}(f,z)&=\sum_ {j}|h^t_ z(a_ j)|^2 \int_{D_ j} |f(u)-\widehat{f_ r}(a_ j)|^2\, dv_{\alpha}(u)\\
\\
&=\sum_ {j}|h^t_ z(a_ j)|^2\,|D_ j|_{\alpha}\,MO_ r(f)(a_ j)^2.
\end{split}
\]
and
\begin{equation}\label{T2-EqA1-2}
A_{12}(f,z)=\sum_ {j}|h^t_ z(a_ j)|^2 \,|D_ j|_{\alpha}\, |\widehat{f_ r}(a_ j)-\widehat{f_ r}(z)|^2.
\end{equation}
In order to estimate $A_{12}(f,z)$ we need the following technical lemma.

\begin{lemma}\label{GLem}
Let $r>0$ and $\{\xi_ m\}$ be an $(r/3)$-lattice on $\Bn$. Let $0<p<\infty$ and $d,\delta>0$. For $a,z\in \Bn$, we have
\[
|\widehat{f_ r}(z)-\widehat{f_ r}(a)| \lesssim  h_{\delta}(a,z) \, N_ p(f,a)^{1/p} \,\,|1-\langle z, a\rangle |^{d},
\]
with
\[
N_ p(f,a)=\sum_ m \frac{MO_{2 r}(f)(\xi_ m)^p\,(1-|\xi_ m|^2)^{\delta p}}{|1-\langle \xi_ m, a\rangle |^{pd}},
\]
and
\[
h_{\delta}(a,z)=\big (1+\beta(a,z)\big ) \,\big [\min (1-|z|,1-|a|)\big ]^{-\delta}.
\]
\end{lemma}

\begin{proof}
Denote by $\gamma(t)$, $0\le t\le 1$, the geodesic in the Bergman metric going from $z$ to $a$.
Let $N=[\beta(z,a)/R]+1$ with $R=r/3$, and $t_ m=m/N$, $0\le m\le N$, where $[x]$ denotes the largest integer less than or equal to $x$.
Set $z_m=\gamma (t_ m)$, $0\le m\le N$. Clearly
$$\beta(z_ m,z_{m+1})=\frac{\beta(z,a)}{N}\le R=r/3.$$
By the triangle inequality, we have
\[
\begin{split}
|\widehat{f_ r}(z)-\widehat{f_ r}(a)|\le \left (\sum _ {m=1}^{N} |\widehat{f_ r}(z_{ m-1} )-\widehat{f_ r}(z_{m})|\right ).
\end{split}
\]
For each $m$, take a point $\xi_ m$ in the lattice with $\beta(z_ m,\xi_ m)<r/3.$ It is not difficult to see that $|\widehat{f_ r}(\xi)-\widehat{f_ r}(w)|\lesssim MO_{2r}(f)(\xi)$ if $\beta(\xi,w)\le r$ (see \cite[p.211]{Zhu}). Then

\[
|\widehat{f_ r}(z_{ m-1} )-\widehat{f_ r}(z_{m})| \le |\widehat{f_ r}(z_{ m-1} )-\widehat{f_ r}(\xi_{m})|
+|\widehat{f_ r}(\xi_ m)-\widehat{f_ r}(z_{m})|\lesssim \,MO_{2 r}(f)(\xi_ m).
\]
This gives
\begin{equation}\label{GL-E1}
|\widehat{f_ r}(z)-\widehat{f_ r}(a)|\lesssim \sum _ {m=1}^{N} MO_ {2r}(f)(\xi_ m).
\end{equation}
Because the M\"{o}bius transformation $\varphi_ z$ sends the geodesic joining $z$ and $a$ to
the geodesic joining $0$ and $\varphi_ z(a)$, we have
\[
|1-\langle \varphi_ z(a),\varphi_ z(z_ m)\rangle |=1-|\varphi_ z(a)|\,|\varphi_ z(z_ m)| \le 1-|\varphi_ z(z_ m)|^2.
\]
Developing this inequality using the basic identity \eqref{BId} together with its polarized analogue \cite[Lemma 1.3]{ZhuBn}
\[
1-\langle \varphi_ z(a),\varphi_ z(b)\rangle=\frac{(1-|z|^2)(1-\langle a,b \rangle )}{(1-\langle a,z \rangle )\,(1-\langle z,b \rangle )}
\]
we arrive at
\[
\frac{|1-\langle a,z_ m \rangle |}{|1-\langle a,z \rangle |}\le \frac{(1-|z_ m|^2)}{|1-\langle z,z_ m\rangle |}.
\]
which gives
\[
|1-\langle a,z_ m  \rangle | \le 2\, |1-\langle a,z\rangle |.
\]
Putting these inequality into \eqref{GL-E1}, with the help of the estimate $|1-\langle \xi_ m, a\rangle |\asymp |1-\langle z_ m, a \rangle |$ (see \cite[p.63]{ZhuBn}), we obtain
\[
|\widehat{f_ r}(z)-\widehat{f_ r}(a)|\lesssim \sum _ {m=1}^{N}\frac{MO_{2r}(f)(\xi_ m)}{|1-\langle \xi_ m, a \rangle |^d}\,\,|1-\langle z,a \rangle |^{d}.
\]
From here, the result easily follows, since
\[
\sum _ {m=1}^{N}\frac{MO_{2r}(f)(\xi_ m)}{|1-\langle \xi_ m, a \rangle |^d}\le \left ( \sum _ {m=1}^{N}\frac{MO_{2r}(f)(\xi_ m)^p}{|1-\langle \xi_ m, a \rangle |^{pd}}\right )^{1/p},\qquad 0<p\le 1,
\]
and H\"{o}lder's inequality yields
\[
\sum _ {m=1}^{N}\frac{MO_{2r}(f)(\xi_ m)}{|1-\langle \xi_ m, a \rangle |^d}\lesssim N^{1/p'}\left ( \sum _ {m=1}^{N}\frac{MO_{2r}(f)(\xi_ m)^p}{|1-\langle \xi_ m, a \rangle |^{pd}}\right )^{1/p},\quad 1<p<\infty.
\]
Finally, since $N\lesssim (1+\beta(a,z))$, the inequality
\[
\min \big (1-|a|,1-|z|\big )\le (1-|z_ m|)\asymp (1-|\xi_ m|)
\]
completes the proof of the lemma.
\end{proof}
\mbox{}
\\
Returning to the estimate for $A_ {12}(f,z)$, putting the inequality of Lemma \ref{GLem} into \eqref{T2-EqA1-2}, with $d=\frac{1}{2}(n+1+\alpha+2t)-\varepsilon$, where $\varepsilon>0$ is taken so that $pd>n$,
we see that $A_ {12}(f,z)$ is less than constant times

\[
(1-|z|^2)^{n+1+\alpha+2t} N_ p(f,z)^{2/p} \sum_ j  \!\! \frac{(1-|a_ j|^2)^{n+1+\alpha}}{|1-\langle a_ j,z \rangle |^{2(n+1+\alpha+t-d)}}\,\, h_{\delta}(a_ j,z)^2,
\]
with $\delta>0$ taken so that $\alpha-2\delta>-1$ and $pd-p\delta>n$.
By Lemma \ref{l2} and Lemma \ref{Ic-discretebeta}, we have
\[
A_ {12}(f,z) \lesssim (1-|z|^2)^{2d-2\delta} N_ p(f,z)^{2/p}.
\]
Then
\begin{equation}\label{EqT2-EB2}
\begin{split}
\int_{\Bn} A_ {12}(f,z)^{p/2} &\,d\lambda_ n(z) \lesssim \int_{\Bn} N_ p(f,z)\,(1-|z|^2)^{p(d-\delta)}d\lambda_ n(z)\\
\\
& =\sum_ m MO_ {2r}(f)(\xi_ m)^p (1-|\xi_ m|^2)^{\delta p} \int_{\Bn}\frac{(1-|z|^2)^{p(d-\delta)}d\lambda_ n(z)}{|1-\langle \xi_ m, z \rangle |^{pd} }\\
\\
& \lesssim \sum_ m MO_ {2r}(f)(\xi_ m)^p,
\end{split}
\end{equation}
after an application of Lemma \ref{IctBn}.\\

 It remains to estimate $\int_{\Bn} A_{11}(f,z)^{p/2}d\lambda_ n(z)$. In case that $0<p\le 2$, then
\[
A_{11}(f,z)^{p/2}\le \sum_ {j}|h^t_ z(a_ j)|^p\,|D_ j|_{\alpha}^{p/2}\,MO_ r(f)(a_ j)^p.
\]
This together with Lemma \ref{IctBn} (due to our condition $p>2n/(n+1+\alpha+2t)$, its application is correct) gives
\[
\begin{split}
\int_{\Bn} &A_{11}(f,z)^{p/2}\, d\lambda_ n(z) \\
\\
&\lesssim \sum_ j (1-|a_ j|^2)^{\frac{p}{2}(n+1+\alpha)} \,MO_ r(f)(a_ j)^p \int_{\Bn} \frac{(1-|z|^2)^{\frac{p}{2}(n+1+\alpha+2t)}}{|1-\langle z,a_ j \rangle |^{p(n+1+\alpha+t)}} d\lambda_ n(z)\\
\\
& \lesssim \sum_ j MO_ r(f)(a_ j)^p.
\end{split}
\]
\mbox{}
\\
If $p>2$, by H\"{o}lder's inequality with exponent $p/2>1$ (denoting its dual exponent by $(p/2)'$) and Lemma \ref{l2}, we see that
\[
\left (\sum_ j \frac{|D_ j|_{\alpha}\, MO_ r(f)(a_ j)^2}{|1-\langle a_ j,z \rangle |^{2(n+1+\alpha+t)}}\right )^{p/2}
\]
is less than
\[
\begin{split}
\sum_ j &\frac{(1-|a_ j|^2)^{\frac{p}{2}(1+\alpha)}\, MO_ r(f)(a_ j)^p}{|1-\langle a_ j,z \rangle |^{\frac{p}{2}(n+2+2\alpha+2t-\varepsilon)}} \left ( \sum_ j \frac{(1-|a_ j|^2)^{n\,(\frac{p}{2})'}}{|1-\langle a_ j,z \rangle |^{(\frac{p}{2})'(n+\varepsilon)}}\right )^{\frac{p}{2}-1}\\
\\
& \lesssim (1-|z|^2)^{-\varepsilon\frac{p}{2}}\,\sum_ j \frac{(1-|a_ j|^2)^{\frac{p}{2}(1+\alpha)}\, MO_ r(f)(a_ j)^p}{|1-\langle a_ j,z \rangle |^{\frac{p}{2}(n+2+2\alpha+2t-\varepsilon)}}.
\end{split}
\]
Hence, for $p>2$, we have
\[
A_{11}(f,z)^{p/2} \lesssim (1-|z|^2)^{\frac{p}{2}(n+1+\alpha+2t-\varepsilon)}\,\sum_ j \frac{(1-|a_ j|^2)^{\frac{p}{2}(1+\alpha)}\, MO_ r(f)(a_ j)^p}{|1-\langle a_ j,z \rangle |^{\frac{p}{2}(n+1+\alpha+t)}}.
\]
Therefore, we obtain

\[
\begin{split}
\int_{\Bn} & A_{11}(f,z)^{p/2}\,d\lambda_ n(z) \\
\\
& \lesssim \sum_ j (1-|a_ j|^2)^{\frac{p}{2}(1+\alpha)}\, MO_ r(f)(a_ j)^p \int_{\Bn} \frac{(1-|z|^2)^{\frac{p}{2}(n+1+\alpha+2t-\varepsilon)}\,d\lambda_ n(z)}{|1-\langle a_ j,z \rangle |^{\frac{p}{2}(n+2+2\alpha+2t-\varepsilon)}}
\\
\\
& \lesssim \sum_ j  MO_ r(f)(a_ j)^p,
\end{split}
\]
because as $\varepsilon>0$ has been taken so that $\frac{p}{2}(n+1+\alpha+2t-\varepsilon)>n$, by Lemma \ref{IctBn}, we have

\[
\int_{\Bn} \frac{(1-|z|^2)^{\frac{p}{2}(n+1+\alpha+2t-\varepsilon)}\,d\lambda_ n(z)}{|1-\langle a_ j,z \rangle |^{\frac{p}{2}(n+2+2\alpha+2t-\varepsilon)}}\lesssim (1-|a_ j|^2)^{ -\frac{p}{2}(1+\alpha)}.
\]
Thus, we have proved that

\begin{equation*}
\int_{\Bn}A_{11}(f,z)^{p/2}\,d\lambda_ n(z)\lesssim \sum_ j  MO_ r(f)(a_ j)^p,\qquad 0<p<\infty.
\end{equation*}
This together with \eqref{EqT2-EB2} and \eqref{T2-EqB0} proves \eqref{Eq-T2-A1} finishing the proof of the theorem.

\section{Proof of Theorem \ref{mt}: first steps}\label{S-4}

We first establish some auxiliary results that can be of independent interest.
Recall that $h_ z^t=K_ z^t/\|K_ z ^t\|_{\alpha}$ with
$
K_ z ^t
$ defined in \eqref{DefKt}. We begin with a tricky lemma.

\begin{lemma}\label{KL}
Let $\alpha>-1,$ $t\ge 0$ and $f\in L^2(\Bn,dv_{\alpha})$. Then
\begin{displaymath}
MO_{\alpha,t}(f)(z) \le C \cdot \big( \|H_ f h^t_ z\|_{\alpha}+\|H_{\overline{f}}\, h^t_ z\|_{\alpha}\big)
\end{displaymath}
for each $z\in \Bn$.
\end{lemma}
\begin{proof}
An easy computation gives
\[ \|(f-\lambda) h^t_ z \|^2_{\alpha}=B_{\alpha,t}(|f|^2)(z)-\big |B_{\alpha,t}(f)(z)\big |^2 +\big |B_{\alpha,t}(f)(z)-\lambda |^2.\]
Thus, we have
\begin{displaymath}
\begin{split}
MO_{\alpha,t}(f)(z) &=\left (B_{\alpha,t}(|f|^2)(z)-\big |B_{\alpha,t}(f)(z)\big |^{2}\right )^{1/2}
\\
&\le \|fh^t_ z-\overline{g_ z(z)}\, h^t_ z\|_{\alpha},
\\
&\le \|fh^t_ z -P_{\alpha}(f h^t_ z)\|_{\alpha}+\|P_{\alpha}(f h^t_ z)-\overline{g_ z(z)}\, h^t_ z\|_{\alpha}
\\
&=\|H_ f h^t_ z\|_{\alpha}+\|P_{\alpha}(f h^t_ z)-\overline{g_ z(z)}\, h^t_ z\|_{\alpha},
\end{split}
\end{displaymath}
where $g_ z$ denotes the holomorphic function on $\Bn$ given by
\[ g_ z(w)=\frac{P_{\alpha}(\overline{f}\, h^t_ z)(w)}{h^t_ z(w)},\qquad w\in \Bn.\]
Now we use the identity
\begin{equation}\label{Eq-H1}
\overline{g_ z(z)} h^t_ z =P_{\alpha+t}(\overline{g_ z}\, h^t_ z).
\end{equation}
To see this, since $K_ z^t(w)=\overline{K_ w^t(z)}$, by the reproducing formula
\begin{displaymath}
\begin{split}
\overline{g_ z(z)} h^t_ z (w)&= \|K^t_ z\|_{\alpha}^{-1} \, \overline{g_ z(z)\,K_ w^t(z)}=\|K^t_ z\|_{\alpha}^{-1}  \, \overline{\langle g_ z K_ w^t, K_ z \rangle _{\alpha}}\\
\\
&=\|K^t_ z\|_{\alpha}^{-1}\langle K_ z, g_ z K_ w^t \rangle _{\alpha}=\|K^t_ z\|_{\alpha}^{-1}\langle K^t_ z, g_ z K_ w ^t \rangle _{\alpha+t}\\
\\
&=\langle \overline{g_ z}\,h^t_ z, K_ w ^t \rangle _{\alpha+t}=P_{\alpha+t}(\overline{g_ z} \,h^t_ z)(w).
\end{split}
\end{displaymath}
Therefore, \eqref{Eq-H1} together with the boundedness of $P_{\alpha+t}$ on $L^2(\Bn,dv_{\alpha})$ yields

\begin{equation}\label{Eq-H6}
\begin{split}
\|P_{\alpha}(f h^t_ z)-\overline{g_ z(z)}\, h^t_ z\|_{\alpha}&=\|P_{\alpha}(f h^t_ z)-P_{\alpha+t}(\overline{g_ z} \,h^t_ z)\|_{\alpha}\\
\\
&=\big \|P_{\alpha+t} \big (P_{\alpha}(f h^t_ z)-\overline{g_ z} \,h^t_ z)\big )\big \|_{\alpha}\\
\\
&\le \|P_{\alpha+t}\| \cdot \|P_{\alpha}(f h^t_ z)-\overline{g_ z} \,h^t_ z\|_{\alpha}.
\end{split}
\end{equation}
Finally,
\begin{displaymath}
\begin{split}
\|P_{\alpha}(f h^t_ z)-\overline{g_ z} \,h^t_ z\|_{\alpha}& \le \|f h^t_ z -P_{\alpha}(f h^t_ z)\|_{\alpha}+ \|fh^t_ z -\overline{g_ z} \,h^t_ z\|_{\alpha}\\
\\
&=\|H_ f h^t_ z\|_{\alpha}+\|\overline{f}\,h^t_ z -g_ z \,h^t_ z\|_{\alpha}\\
\\
&=\|H_ f h^t_ z\|_{\alpha}+\|\overline{f}\,h^t_ z -P_{\alpha}(\overline{f}\,h^t_ z)\|_{\alpha}\\
\\
&=\|H_ f h^t_ z\|_{\alpha}+\|H_{\overline{ f}}\, h^t_ z\|_{\alpha}.
\end{split}
\end{displaymath}
This proves the result with constant $C=(1+\|P_{\alpha+t}\|)$. Observe that, when $t=0$, since $\|P_{\alpha}\|=1$, one gets $C=1$ since in \eqref{Eq-H6} one has the term $\|P_{\alpha}(f k_ z-\overline{g_ z} \,k_ z)\|_{\alpha}$, and thus it is not necessary to use again the triangle inequality.
\end{proof}
\mbox{}
\\
The case $t=0$ of Lemma \ref{KL} appears in \cite{BBCZ} and \cite{Z1}, with a proof that seems to be specific of the Hilbert space case. Observe that our proof is flexible enough to work when studying Hankel operators acting on $A^p_{\alpha}$ (see \cite{Zhu-Pac}, where some version of Lemma \ref{KL} for $t=0$ in this setting was proved with a different method).\\

The following inequality is also satisfied:
\begin{equation}\label{TE}
\|H_ f h^t_ z\|+\|H_{\overline{f}}\, h^t_ z\| \le 2 \,MO_{\alpha,t}(f)(z).
\end{equation}
Indeed, we have
\[\|H_ f h^t_ z\|_{\alpha}^2 =\|(I-P_{\alpha})(fh^t_ z)\|^2_{\alpha}=\|fh^t_ z\|^2_{\alpha}-\|P_{\alpha}(fh^t_ z)\|^2_{\alpha}.\]
Now, by Cauchy-Schwarz,
\[ |B_{\alpha,t}(f)(z)|=|\langle f h^t_ z,h^t_ z \rangle _{\alpha} |=|\langle P_{\alpha}(f h^t_ z),h^t_ z \rangle _{\alpha} |\le \|P_{\alpha}(fh^t_ z)\|_{\alpha}. \]
Since $\|fh^t_ z\|^2_{\alpha}=B_{\alpha,t}(|f|^2)(z)$, it follows that $\|H_ f h^t_ z\|_{\alpha}\le MO_{\alpha,t}(f)(z)$.\\

The following result can be found  in \cite[Lemma 2]{Pau}.
\begin{otherl}\label{KL-1}
Let $\alpha>-1$ and $T:A^2_{\alpha}(\Bn)\rightarrow A^2_{\alpha}(\Bn)$ be a positive operator. For $t\ge 0$ set
\[\widetilde{T^t}(z)=\langle Th^{t}_ z,h^{t}_ z\rangle_{\alpha},\quad z\in \Bn.\]
\begin{enumerate}
\item[(a)]  Let $0<p\le 1$. If $\widetilde{T^t} \in L^p(\Bn,d\lambda_ n)$ then $T$ is in $S_ p$.

\item[(b)] Let $p\ge 1$. If $T$ is in $S_ p$ then $\widetilde{T^t} \in L^p(\Bn,d\lambda_ n)$.
\end{enumerate}
\end{otherl}
If we apply this lemma with the positive operator $T=H_ f^* H_ f$, then due to \eqref{TE} and Lemma \ref{KL}, we obtain  the necessity in Theorem \ref{mt3} for $p\ge 2$ and the sufficiency for $p\le 2$. This together with the inequality $MO_ r(f)(z)\lesssim MO_{\alpha,t}(f)(z)$  gives the implication (a) implies (b) in Theorem \ref{mt} for $p\ge 2$, and if we use Theorem \ref{BMOp} we see that (b) implies (a) for $p\le 2$. Summarizing, the following proposition has been proved.

\begin{proposition}\label{P-mt1}
Let $\alpha>-1$ and $f\in L^2_{\alpha}$. Then
\begin{enumerate}
\item[(i)] Let $2\le p<\infty$. If the Hankel operators $H_ f$ and $H_{\overline{f}}$ are simultaneously in $S_ p(A^2_{\alpha},L^2_{\alpha})$, then $MO_ r(f)$ is in $L^p(\Bn,d\lambda_ n)$.

\item[(ii)] Let $0<p\le 2$. If $MO_ r(f)\in L^p(\Bn,d\lambda_ n)$ then $H_ f$ and $H_{\overline{f}}$ are both in $S_ p(A^2_{\alpha},L^2_{\alpha})$.
\end{enumerate}
\end{proposition}
\mbox{}
\\
Next, we consider the Hankel operator $H_ f ^{\gamma}$ defined by
\[ H_ f^{\gamma} =(I-P_{\gamma})M_ f.\]
With this notation, we have $H_ f=H_ f^{\alpha}$.

\begin{lemma}\label{KL-3}
Let $\alpha>-1$,  $f\in L^2_{\alpha}$ and $\gamma>\alpha$. If $H_ f^{\alpha}$ is in $S_ p(A^2_{\alpha},L^2_{\alpha})$ (or compact), then
$H_ f^{\gamma}$ is also in $S_ p(A^2_{\alpha},L^2_{\alpha})$ (or compact).
\end{lemma}

\begin{proof}
Since for $\gamma >\alpha$, the projection $P_{\gamma}$ is bounded on $L^2_{\alpha}$ and $P_{\gamma}P_{\alpha}=P_{\alpha}$, we have
\[
H_ f^{\gamma}=H_ f^{\alpha}+ (P_{\alpha}-P_{\gamma})M_ f =H_ f^{\alpha}-P_{\gamma} H_ f^{\alpha}.
\]
Hence the result follows.
\end{proof}

\begin{proposition}\label{Pr-Suf}
Let $\alpha>-1$, $f\in L^2_{\alpha}$ and $2<p<\infty$. If $MO_ {2r}(f)\in L^p(\Bn,d\lambda_ n)$ then $H_ f$ and $H_{\overline{f}}$ are both in $S_ p(A^2_{\alpha},L^2_{\alpha})$.
\end{proposition}

\begin{proof}
 This follows from Theorem \ref{BMOp} with $t=0$ and the well know fact that $MO_{\alpha,0}(f)\in L^p(\Bn,d\lambda_ n)$ implies the conclusion of the Proposition. However, we will provide a self-contained proof based on Lemma \ref{GLem}. Since $MO_ {2r}(\overline{f})=MO_ {2r}(f)$ it suffices to prove that $H_ f$ is in $S_ p(A^2_{\alpha},L^2_{\alpha})$. Write $f=f_ 1+f_ 2$ with
 \[
 f_ 1=f-\widehat{f_ r},\qquad f_ 2=\widehat{f_ r}
 \]
 and proceed to show that both $H_{f_ 1}$ and $H_{f_ 2}$ are in $S_ p$. From \cite[p.211]{Zhu} we have
 \[
\widehat{|f_ 1|^2_ r}(z)\lesssim MO_{2r}(f)(z)^2.
 \]
Hence, by \cite[Theorem 1]{Zhu1} (see also \cite[Corollary 7.14]{Zhu}), the Toeplitz operator $T_{|f_ 1|^2}$ belongs to $S_{p/2}(A^2_{\alpha})$. As
\[
H_{f_ 1} H_{f_ 1}^*= T_{|f_ 1|^2}-T_{f_ 1}T_{f_ 1}^* \le T_{|f_ 1|^2}
\]
this implies that the Hankel operator $H_{f_ 1}$ is in $S_ p(A^2_{\alpha},L^2_{\alpha})$. \\

Next we are going to show that $H_ {f_ 2}$ is in $S_ p$.
 Note that the condition implies that  $H_ {f_ 2}^{\alpha}$  is compact (just take a look at Lemma \ref{DMOr} which implies $MO_ r(f)(z)\rightarrow 0$ as $|z|\rightarrow 1$ and this easily implies that $MO_ r(\widehat{f_ r})(z)\rightarrow 0$ ); and in view of Lemma \ref{KL-3}, the operator $H_ {f_ 2}^{\gamma}$ is also compact for all $\gamma>\alpha$. Since $P_{\gamma}=P_{\alpha}P_{\gamma}$ on $L^2_{\alpha}$ and
 \[
 H_ {f_ 2} ^{\alpha}-H_ {f_ 2}^{\gamma}=(P_{\gamma}-P_{\alpha})M_ {f_ 2} =P_{\alpha}H_ {f_ 2}^{\gamma},
 \]
 it is enough to show that $H_ {f_ 2}^{\gamma}$ belongs to $S_ p(A^2_{\alpha},L^2_{\alpha})$ for $\gamma$ big enough, say $\gamma=\alpha+4t$ with $pt>n$. By \cite[Theorem 1.33]{Zhu}, it suffices to prove that
\[
\sum _ m \big \|H^{\gamma}_ {f_ 2} e_ m \big \|^p_{\alpha}\le C
\]
for any orthonormal set $\{e_ m\}$ of $A^2_{\alpha}$, with a constant $C$ not depending on the choice of the orthonormal set. Let $\varepsilon>0$ so that $\alpha-\varepsilon>-1$. By Cauchy-Schwarz and Lemma \ref{IctBn}, we have
\[
\begin{split}
\|H^{\gamma}_ {f_ 2} e_ m \big \|^2_{\alpha}&\le  \int_{\Bn} \left (\int_{\Bn}\frac{|f_ 2(z)-f_ 2(w)|\,|e_ m(w)|}
{|1-\langle z,w \rangle |^{n+1+\gamma}}\,dv_{\gamma}(w)\right )^2\, dv_{\alpha}(z)
\\
& \lesssim  \int_{\Bn} \left (\int_{\Bn}\frac{|\widehat{f_ r}(z)-\widehat{f_ r}(w)|^2 \,|e_ m(w)|^2}
{|1-\langle z,w \rangle |^{n+1+\gamma}}\,dv_{\gamma+\varepsilon}(w)\right )\, dv_{\alpha-\varepsilon}(z).
\end{split}
\]
Now, Fubini's theorem, H\"{o}lder's inequality with exponent $p/2>1$ and $\|e_ m\|_{\alpha}=1$ yield
\[
\|H^{\gamma}_ {f_ 2} e_ m \big \|^p_{\alpha}\lesssim \int_{\Bn} |e_ m(w)|^2 \left (\int_{\Bn}\frac{|\widehat{f_ r}(z)-\widehat{f_ r}(w)|^2 \,dv_{\alpha-\varepsilon}(z)}
{|1-\langle z,w \rangle |^{n+1+\gamma}}\right )^{p/2}\, dv_{\alpha+\frac{p}{2}(\gamma-\alpha+\varepsilon)}(w).
\]
Because $\{e_ m\}$ is an orthonormal set, we can use the inequality
\[
\sum_ m |e_ m(w)|^2 \le \|K_ w \|^2_{\alpha}=(1-|w|^2)^{-(n+1+\alpha)}
\]
to obtain
\[
\sum_ m \|H^{\gamma}_ {f_ 2} e_ m \big \|^p_{\alpha}\lesssim \int_{\Bn} \!\left (\int_{\Bn}\frac{|\widehat{f_ r}(z)-\widehat{f_ r}(w)|^2 \, dv_{\alpha-\varepsilon}(z)}
{|1-\langle z,w \rangle |^{n+1+\gamma}}\right )^{p/2}\! \! (1-|w|^2)^{\frac{p}{2}(\gamma-\alpha+\varepsilon)}\,d\lambda_ n(w).
\]
Set
\[
I_ p(f,w):=\left (\int_{\Bn}\frac{|\widehat{f_ r}(z)-\widehat{f_ r}(w)|^2 \, dv_{\alpha-\varepsilon}(z)}
{|1-\langle z,w \rangle |^{n+1+\gamma}}\right )^{p/2}.
\]
Take a lattice $\{\xi_ k\}$ and apply Lemma \ref{GLem} with $d=t$ and $\delta>0$ satisfying $pt-p\delta>n$ and $\alpha-\varepsilon-2\delta>-1$, to obtain
\[
I_ p(f,w) \lesssim N_ p(f,w) \,\left (\int_{\Bn}\frac{ h(z,w)\,dv_{\alpha-\varepsilon}(z)}
{|1-\langle z,w \rangle |^{n+1+\gamma-2d}}\right )^{p/2}
\]
with
\[
N_ p(f,w)=\sum_ k \frac{MO_{2 r}(f)(\xi_ k)^p\,(1-|\xi_ k|^2)^{\delta p}}{|1-\langle w,\xi_ k\rangle |^{pd}}
\]
and
\[
h(z,w)=\big [ 1+\beta(z,w)\big ]^{2} \big (\min (1-|z|,1-|w|)\big )^{-2\delta}.
\]
By Lemma \ref{IctBn} and Lemma \ref{Itbeta}, we have
\[
\int_{\Bn}\frac{ h(z,w)\,dv_{\alpha-\varepsilon}(z)}
{|1-\langle z,w \rangle |^{n+1+\gamma-2d}} \lesssim (1-|w|^2)^{2d-\gamma+\alpha-\varepsilon-2\delta}.
\]
This, together with Lemma \ref{IctBn} gives
\[
\begin{split}
\sum_ m \|H^{\gamma}_ {f_ 2} e_ m \big \|^p_{\alpha}&\lesssim \int_{\Bn} I_ p(f,w)  \,(1-|w|^2)^{\frac{p}{2}(\gamma-\alpha+\varepsilon)}\,d\lambda_ n(w)
\\
&\lesssim \int_{\Bn} N_ p(f,w)\,(1-|w|^2)^{pd-p\delta} \,d\lambda_ n(w)
\\
&=\sum_ k MO_ {2r}(f)(\xi_ k)^p\,(1-|\xi_ k|^2)^{\delta p} \int_{\Bn} \frac{(1-|w|^2)^{pd-p\delta}}{|1-\langle w,\xi_ k\rangle |^{pd}} \,d\lambda_ n(w)
\\
&\lesssim \sum_ k MO_{2 r}(f)(\xi_ k)^p.
\end{split}
\]
In view of Lemma \ref{DMOr} this finishes the proof.
\end{proof}
\mbox{}
\\

Taking into account Propositions \ref{P-mt1} and \ref{Pr-Suf}, in order to complete the proof of Theorem \ref{mt} it remains to show that (a) implies (b) for $0<p<2$. The case $\frac{2n}{n+1+\alpha}<p<2$ follows immediately from  (c) and the fact that $MO_ r(f)(z)\lesssim MO_{\alpha}(f)(z)$. The final case is done in the next  section.

\section{The last case: necessity for $0<p\le \frac{2n}{n+1+\alpha}$}\label{S-6}

In order to prove this case, we will fix a number $\beta>\alpha$ satisfying $p(n+1+\beta)>2n$. We will show that condition (a) of Theorem \ref{mt} implies that both $H_ f^{\beta}$ and $H_{\overline{f}}^{\beta}$ are in $S_ p(A^2_{\beta},L^2_{\beta})$. Then the case already proved will give $MO_ r(f)\in L^p(\Bn,d\lambda_ n)$. \\

We will use that, under the pairing $\langle \cdot,\cdot \rangle _{\gamma}$ with $\gamma=(\alpha+\beta)/2$, the dual of $L^2_{\alpha}$ can be identified with $L^2_{\beta}$. Thus, if $T$ is an operator in $L^2_{\alpha}$, we can consider its adjoint operator $S$ respect to the pairing $\langle \cdot,\cdot \rangle _{\gamma}$ (acting now on $L^2_{\beta}$) defined by the relation
\begin{equation}\label{OpS}
\langle T u, v \rangle _{\gamma}=\langle u, Sv \rangle _{\gamma}, \qquad u\in L^2_{\alpha}, \quad v\in L^2_{\beta}.
\end{equation}

\begin{lemma}\label{LC-L1}
Let $T\in S_ p(L^2_{\alpha})$. Then the operator $S$ defined by \eqref{OpS} is in $S_ p(L^2_{\beta})$. Moreover $\|T\|_{S_ p}\asymp  \|S\|_{S_ p}$.
\end{lemma}

\begin{proof}
Let
\[
Tu=\sum_ n \lambda_ n \,\langle u,e_ n \rangle _{\alpha} \,\sigma_ n ,\qquad u\in L^2_{\alpha}
\]
be the canonical decomposition of the operator $T$, where $\{e_ n\}$ and $\{\sigma_ n\}$ are orthonormal sets of $L^2_{\alpha}$, and $\{\lambda_ n\}$ are the singular values of $T$.  For each $n$, consider the functions
\[
f_ n(z)=e_ n(z)\,(1-|z|^2)^{\alpha-\gamma}\qquad  \textrm{and}\qquad  h_ n(z)=\sigma_ n(z)\,(1-|z|^2)^{\alpha-\gamma}.
\]
Then $\{f_ n\}$ and $\{h_ n\}$ are orthogonal sets in $L^2_{\beta}$, with $\|f_ n\|_{\beta}=\|h_ n\|_{\beta}=\sqrt{c_{\beta}/c_{\alpha}}$, where $c_{\alpha}$ is the normalizing constant appearing in the definition of $dv_{\alpha}$. Also
\[
\langle u,e_ n \rangle_{\alpha}=K_{\alpha,\gamma}\,\langle u, f_ n \rangle _{\gamma}
\]
with $K_{\alpha,\gamma}=c_{\alpha}/c_{\gamma}.$ Then it follows that
\[
Sv= K_{\alpha,\gamma} \sum_ n \overline{\lambda_ n} \,\, \langle v,\sigma_ n \rangle _{\gamma}\, f_ n ,\qquad v\in L^2_{\beta}.
\]
Since $\langle \sigma_ n , v \rangle _{\gamma}=(c_{\gamma}/c_{\beta})\,\langle h_ n, v \rangle _{\beta}$, normalizing the functions $f_ n$ and $h_ n$ in $L^2_{\beta}$ we see that $\{ \overline{\lambda_ n}\}$ are the singular values of the operator $S$ acting on $L^2_{\beta}$. This gives the result.
\end{proof}

\begin{lemma}
Suppose that $H_{f}^{\alpha}$ and $H_{\overline{f}}^{\alpha}$ are both in $S_ p(A^2_{\alpha},L^2_{\alpha})$. Then the commutator $[M_ f,P_{\gamma}]$ is in $S_ p(L^2_{\alpha})$.
\end{lemma}

\begin{proof}
It is enough to show that $[M_ f,P_{\gamma}]-[M_ f,P_{\alpha}]$ is in $S_ p(L^2_{\alpha})$. Some algebraic manipulations give
\[
\begin{split}
[M_ f,P_{\gamma}]-[M_ f,P_{\alpha}]&=M_ f P_{\gamma}-P_{\gamma}M_ f -M_ f P_{\alpha}+P_{\alpha}M_ f \\
\\
&=M_ f(P_{\gamma}-P_{\alpha})-\widetilde{H_ f^{\gamma}}-P_{\gamma}M_ f P_{\gamma}+\widetilde{H_ f^{\alpha}}+P_{\alpha}M_ f P_{\alpha}.
\end{split}
\]
Here $\widetilde{H_ f^{s}}=(I-P_ s)M_ f P_ s$. We already know that $\widetilde{H_ f^{\alpha}}$ is in $S_ p(L^2_{\alpha})$, and by Lemma \ref{KL-3} we also have $\widetilde{H_ f^{\gamma}}\in S_ p(L^2_{\alpha})$ because $P_{\gamma}:L^2_{\alpha}\rightarrow A^2_{\alpha}$ is bounded. Thus, it is enough to see that the operator
\[
T:=M_ f(P_{\gamma}-P_{\alpha})-P_{\gamma}M_ f P_{\gamma}+P_{\alpha}M_ f P_{\alpha}
\]
is in $S_ p(L^2_{\alpha})$. Since $P_{\gamma}=P_{\alpha}P_{\gamma}$ and $P_{\alpha}=P_{\gamma}P_{\alpha}$ on $L^2_{\alpha}$, we have
\[
\begin{split}
T&=(I-P_{\alpha})M_ f (P_{\gamma}-P_{\alpha})+(P_{\alpha}-P_{\gamma})M_ f P_{\gamma}\\
&=\widetilde{H_ f^{\alpha}}(P_{\gamma}-I)-P_{\gamma}\widetilde{H_ f^{\alpha}} P_{\gamma}.
\end{split}
\]
This shows that $T$ is in $S_ p(L^2_{\alpha})$ finishing the proof.
\end{proof}
\mbox{}
\\
Now that we know that the commutator $T=[M_ f,P_{\gamma}]$ is in $S_ p(L^2_{\alpha})$, an application of Lemma \ref{LC-L1} gives that its adjoint $S$ respect to the pairing $\langle \cdot,\cdot \rangle_{\gamma}$ is in $S_ p(L^2_{\beta})$. A simple computation gives $S=-[M_{\overline{f}},P_{\gamma}]$. Since $P_{\gamma}$ is bounded on $L^2_{\beta}$ (from \cite[Theorem 2.11]{ZhuBn} we have that $P_ s$ is bounded on $L^2_{\sigma}$ if and only if $2(1+s)>1+\sigma$. In our case, $\sigma=\beta$ and $s=\gamma =(\alpha+\beta)/2$, so that we get the condition $\alpha>-1$), then $\widetilde{H_{\overline{f}}^{\gamma}}=[M_{\overline{f}},P_{\gamma}] P_{\gamma}$ is also in $S_ p(L^2_{\beta})$. Hence $H_{\overline{f}}^{\gamma}$ belongs to $S_ p(A^2_{\beta},L^2_{\beta})$, and finally we will show that this implies $H_{\overline{f}}^{\beta}$ in $S_ p(A^2_{\beta},L^2_{\beta})$. To see this it is enough to prove that $H_{\overline{f}}^{\beta}-H_{\overline{f}}^{\gamma}$ is in $S_ p(A^2_{\beta},L^2_{\beta})$, but using that $P_{\gamma}=P_{\beta}P_{\gamma}$, we have
\[
H_{\overline{f}}^{\beta}-H_{\overline{f}}^{\gamma}=(P_{\gamma}-P_{\beta})M_ {\overline{f}} =-P_{\beta} H^{\gamma}_{\overline{f}},
\]
and the result follows. In the same manner we also have $H_{f}^{\beta}$ in $S_ p(A^2_{\beta},L^2_{\beta})$. This completes the proof of Theorem \ref{mt}.

\section{Further remarks}
One can also consider the problem of describing the simultaneous membership of $H_ f^{\alpha}$ and $H_{\overline{f}}^{\alpha}$ in $S_ p(L^2_{\beta},A^2_{\alpha})$, that is, when the weights are not necessarily the same, in the lines of the results of Janson and Wallst\'{e}n \cite{J,Wall} in the holomorphic case. The result that must be obtained following the proof given here is that $H_ f^{\alpha}$ and $H_{\overline{f}}^{\alpha}$ are both in $S_ p(A^2_{\beta},L^2_{\alpha})$ if and only if the function $(1-|z|^2)^{\gamma} MO_ r(f)(z)$ is in $L^p(\Bn,d\lambda_ n)$, with $\gamma=(\alpha-\beta)/2$. The general form of Theorem \ref{BMOp} that can be proved is: let $\gamma \in \mathbb{R}$ with $2\gamma<1+\alpha$ and $0<p<\infty$. Then, for each $t\ge 0$ such that  $p>2n/(n+1+\alpha+\gamma+2t)$, one has
\begin{displaymath}
\int_{\Bn} (1-|z|^2)^{\gamma p} \,MO_{\alpha,t}(f)(z)^p\,d\lambda_ n (z) \le C \int_{\Bn} (1-|z|^2)^{\gamma p}\, MO_{r}(f)(z)^p\,d\lambda_ n (z).
\end{displaymath}
The proof, as well as the other analogues needed, seems to be essentially the same but more technical in the sense that more parameters are involved.
\\
\mbox{}
\\
\textbf{Acknowledgements:} Part of this project was done when the author visited the School of Mathematics at the  Aristotle University of Thessaloniki on April 2014. The author thanks the institution for the great atmosphere and good working conditions received during this period, and he is especially grateful to Petros Galanopoulos for helpful discussions concerning this topic.


\begin{thebibliography}{99}

\bibitem{AFP} J. Arazy,  S. Fisher \and  J. Peetre, \emph{Hankel operators on weighted Bergman
spaces}, Amer. J. Math. 110 (1988), 989--1054.

\bibitem{AFJP}  J. Arazy, S. Fisher, S. Janson \and  J. Peetre, \emph{Membership of Hankel operators on the ball in unitary ideals}, J. London Math. Soc. 43 (1991), 485--508.

\bibitem{BBCZ} D. B\'{e}koll\'{e}, C. Berger, L. Coburn \and K. Zhu, \emph{$BMO$ in the Bergman metric on bounded symmetric domains}, J. Funct. Anal. 93 (1990), 310--350.

\bibitem{Isr} J. Isralowitz, \emph{Schatten $p$ class commutators on the weighted Bergman space $L^2_ a(\Bn,dv_{\gamma})$ for $2n/(n+1+\gamma)<p<\infty$}, Indiana Univ. Math. J. 62 (2013), 201--233.

\bibitem{J} S. Janson, \emph{Hankel operators between weighted Bergman spaces},  Ark. Mat. 26 (1988), 205--219.


\bibitem{M} C.A. McCarthy, \emph{$c_ p$}, Israel J. Math. 5 (1967), 249--271.

\bibitem{Pau} J. Pau, \emph{A remark on Schatten class Toeplitz operators on Bergman spaces}, Proc. Amer. Math. Soc. 142 (2014), 2763--2768.



\bibitem{Wall} R. Wallst\'{e}n, \emph{Hankel operators between weighted Bergman spaces in the ball},  Ark. Mat.  28 (1990), 183--192.

\bibitem{Xia1} J. Xia, \emph{Hankel operators in the Bergman spaces and Schatten $p$-classes: the case $1<p<2$}, Proc. Amer. Math. Soc. 129 (2001), 3559--3567.

\bibitem{Xia2} J. Xia, \emph{On the Schatten class membership of Hankel operators on the unit ball}, Illinois J. Math. 46 (2002), 913--928.

\bibitem{Xia3} J. Xia, \emph{Bergman commutators and norm ideals}, J. Funct. Anal. 263 (2012), 988--1039.

\bibitem{Z0} K. Zhu, \emph{Hilbert-Schmidt Hankel operators on the Bergman space}, Proc. Amer. Math. Soc. 109 (1990), 721--730.

\bibitem{Z1} K. Zhu, \emph{Schatten class Hankel operators on the Bergman space of the unit ball}, Amer J. Math. 113 (1991), 147--167.

\bibitem{Zhu-Pac} K. Zhu, \emph{$BMO$ and Hankel operators on Bergman spaces}, Pacific J. Math. 155 (1992), 377--395.

\bibitem{ZhuBn} K. Zhu, `Spaces of Holomorphic Functions in the Unit Ball', Springer-Verlag, New York, 2005.

\bibitem{Zhu1} K. Zhu,  {\em Schatten class Toeplitz operators on weighted Bergman spaces of the unit ball}, New York J. Math. 13 (2007), 299--316.

\bibitem{Zhu} {K. Zhu}, {`Operator Theory in Function Spaces'}, Second Edition,
Math. Surveys and Monographs 138, American Mathematical
Society: Providence, Rhode Island, 2007.
\end{thebibliography}
\end{document}